\documentclass[11pt]{amsart}

\usepackage[english]{babel}
\usepackage{enumitem}
\usepackage[ansinew]{inputenc}
\usepackage{amsmath,amsthm}
\usepackage{hyperref}
\usepackage[T1]{fontenc}
\usepackage{lmodern}
\usepackage{cleveref}
\usepackage{amssymb}
\usepackage{mathtools}
\usepackage{amscd}
\usepackage{marginnote}
\usepackage[all]{xy} 
\usepackage{dsfont}
\usepackage{tikz-cd}
\usepackage{stmaryrd} 

\setenumerate[1]{label=(\alph*)}
\setenumerate[2]{label=(\roman*)}

\theoremstyle{plain}
\newtheorem{thm}{Theorem}[section]
\newtheorem*{thm*}{Theorem}
\newtheorem{lem}[thm]{Lemma}
\newtheorem{cor}[thm]{Corollary}

\theoremstyle{remark}
\newtheorem{rem}[thm]{Remark}

\theoremstyle{definition}

\newcommand{\ZZ}{\mathbb{Z}}
\newcommand{\QQ}{\mathbb{Q}}
\newcommand{\lcm}{\mathrm{lcm}}

\title[Infinitely many odd zeta values are irrational. By elementary means]{Infinitely many odd zeta values are irrational. By elementary means}

\author{Johannes Sprang}
\thanks{The author has been supported by the SFB 1085 ``Higher Invariants'' funded by the DFG}
\address{Fakult\"at f\"ur Mathematik Universit\"at Regensburg  \\ 93040 Regensburg }
\email{johannes.sprang@mathematik.uni-regensburg.de }
\date{}

\hypersetup{
pdfauthor = {J. Sprang},
pdftitle = {Infinitely many odd zeta values are irrational by elementary means}
}

\begin{document}
\begin{abstract}
In this small note, we provide an elementary proof of the fact that infinitely many odd zeta values are irrational. For the first time, this celebrated theorem been proven by Rivoal and Ball--Rivoal. The original proof uses highly non-elementary methods like the saddle-point method and Nesterenko's linear independence criterion. Recently, Zudilin has re-proven a slightly weaker form of his important result that at least one of the odd zeta values $\zeta(5),\zeta(7),\zeta(9)$ and $\zeta(11)$ is irrational, by elementary means. His new main ingredient are certain 'twists by half' of hypergeometric series. Generalizing this to 'higher twists' allows us to give a purely elementary proof of the result of Rivoal and Ball--Rivoal.
\end{abstract}

\maketitle
The question about the values of the \emph{Riemann zeta function} 
\[
	\zeta(s)=\sum_{n=1}^\infty \frac{1}{n^s}
\]
at the positive integers is very old and has its origins in the 17th century. The first breakthrough has been made by Euler who showed the remarkable formulas
\[
	\zeta(2)=\frac{\pi^2}{6},\quad \zeta(4)=\frac{\pi^4}{90},\quad \zeta(6)=\frac{\pi^6}{945}, ...
\]
More generally, he has proven
\[
	\zeta(2n)=\frac{(-1)^{n+1}B_{2n}(2\pi)^{2n}}{2(2n)!},
\]
where $B_n$ denotes the $n$-th Bernoulli number. Combined with Lindemann's proof of the transcendence of $\pi$, this proves the transcendence of all even zeta values. While we have a good understanding of the structure of the even zeta values, almost nothing is known for the odd zeta values $\zeta(2n+1)$. The first non-trivial result on the structure of $\zeta(2n+1)$ has been obtained by Ap\'ery in 1979. He was able to proof that $\zeta(3)$ is irrational. Although this is the only odd zeta value for which we know the irrationality for sure, the best known result aiming towards the irrationality of $\zeta(5)$ is due to Zudilin. He was able to prove the amazing result that at least one of the numbers
\[
	\zeta(5),\zeta(7),\zeta(9),\zeta(11)
\]
is irrational. Further, we know by a celebrated theorem of Rivoal and Ball--Rivoal, that infinitely many among the odd zeta values
\[
	\zeta(3),\zeta(5),\zeta(9),...
\]
are irrational, \cite{rivoal},\cite{ball_rivoal}. While the original proofs of the results of Zudilin, Rivoal and Rivoal--Ball are highly non-elementary, Zudilin has recently presented a very elementary proof of the fact that at least one of the zeta values
\[
	\zeta(5),\zeta(7),...,\zeta(25)
\]
is irrational. He explicitly constructs a sequence of non-zero linear forms
\[
	\hat{r}_n=a_{0,n}+a_{5,n}\zeta(5)+a_{7,n}\zeta(7)+...+a_{25,n}\zeta(25),\quad a_{ij}\in\ZZ
\]
tending to zero as $n$ goes to infinity. If all involved zeta values were irrational, we could multiply by their common denominator and we would obtain a non-zero sequence of integers converging to zero - an obvious contradiction. The new key ingredient for the construction of the involved linear forms is that he introduces certain 'twists by half' of well-poised hypergeometric series. Comparing both, the twisted and the untwisted, hypergeometric series allows him to eliminate the unwanted value $\zeta(3)$. The present work builds heavily on Zudilin's beautiful and elementary proof. Without his paper \cite{zudilin}, this paper would probably not exist.\par 
In this work, we introduce 'higher twists' by integers $D>1$ of certain hypergeometric series. This allows us to construct certain linear-forms 
\[
	\sum_{\substack{i=3 \\ i \text{ odd}}}^s a^{(D)}_i\zeta\left(i,\frac{j}{D}\right)+a^{(D,j)}_0.
\]
for special values of the Hurwitz zeta function
\[
	\zeta\left(s,\alpha\right):=\sum_{n=0}^\infty\frac{1}{(n+\alpha)^s}.
\]
By using the formula
\[
	(D^i-1)\cdot\zeta(i)=\sum_{j=1}^{D-1}\zeta\left(i,\frac{j}{D}\right)
\]
and a thorough analysis of the convergence properties of the involved linear forms, we can eliminate arbitrary many odd zeta values. In order to guarantee convergence, the parameter $s$ which specifies the range of the zeta values has to be chosen appropriately. \par 
Let us finally remark, that Rivoal's result is stronger in the sense that it provides an asymptotic lower bound for the dimension of the $\QQ$ vector space spanned by
\[
	\{\zeta(3),\zeta(5),...,\zeta(2n+1)\}.
\]
Nevertheless, our proof uses only very elementary techniques. As in Zudilin's proof, the technical ingredients of our proof are limited to Stirling's formula and the asymptotic formula
\[
	\lim_{n\rightarrow \infty} \sqrt[n]{\lcm\{1,...,n\}}=e
\]
which can be deduced from the prime number theorem.\par

\section{Integrality results}
A rational function over the complex numbers
\[
	S(t)=\frac{P(t)}{(t-\alpha_1)^{s_1}(t-\alpha_2)^{s_2}\cdot...\cdot(t-\alpha_r)^{s_r}}\in\mathbb{C}(t)
\]
with $\deg P\leq s:=\sum_{i=0}^r s_i$ has a unique partial-fraction decomposition
\[
	S(t)=\sum_{k=1}^r\sum_{i=1}^{s_k}\frac{b_{i,k}}{(t-\alpha_k)^i}.
\]
The coefficients can be computed explicitly by the formula
\[
	b_{i,k}=\left.\frac{1}{(s_k-i)!}\frac{\partial^{(s_k-i)}}{\partial t^{(s_k-i)}}\left( S(t)(t-\alpha_k)^{s_k} \right)\right|_{t=\alpha_k}.
\]
Let us write $d_n:=\lcm\{1,...,n\}$. The following Lemma can be proven using only the above explicit formula for the coefficients of the partial-fraction decomposition and the Leibniz rule.
\begin{lem}[{\cite[Lemma 1]{zudilin}}]\label{lem1}
 Let $k_1,...,k_r$ be pairwise distinct numbers from the set $\{0,...,n\}$ and $s_1,...,s_r$ positive integers. Then, the coefficients in the expansion
 \[
 	\frac{1}{\prod_{j=1}^r(t+k_j)^{s_j}}=\sum_{j=1}^r\sum_{i=1}^{s_j}\frac{b_{i,j}}{(t+k_j)^i}
 \]
 satisfy
 \[
 	d_n^{s-i}b_{i,j}\in \ZZ, \text{ for } i=1,...,s_j \text{ and } j=1,...,r
 \]
 where $s:=s_1+...+s_r$.
\end{lem}
\begin{proof}
We refer to \cite[Lemma 1]{zudilin} for the proof.
\end{proof}
We will also need the following
\begin{lem}\label{lem2}
	Let $D,n,k,i\in\ZZ$ with $n,D\geq 1$. Then $\frac{n!}{\gcd(D^n,n!)}$ divides $\prod_{j=k}^{n+k-1}(Dj+i)$.
\end{lem}
\begin{proof}
First, let us observe that $\frac{n!}{\gcd(D^n,n!)}$ is the maximal divisor of $n!$ which is co-prime to $D$. Let $p\leq n$ be a prime which is co-prime to $D$. By the Chinese remainder theorem, at least $\lfloor\frac{n}{p}\rfloor$ of the numbers
\[
	\{Dk+i,D(k+1)+i,...,D(k+n-1)+i\}
\]
are divisible by $p$, at least $\lfloor\frac{n}{p^2}\rfloor$ of them are divisible by $p^2$, etc. Let us write $\nu_p(m)$ for the $p$-adic valuation, i.e.~for the exponent of the maximal $p$-power $p^{\nu_p(m)}$ dividing $m$. By the above observation, we obtain
\[
	\sum_{l\geq 1} \lfloor\frac{n}{p^l}\rfloor\leq \nu_p\left(\prod_{j=k}^{n+k-1}(Dj+i)\right).
\]
On the other hand, we have
\[
	\nu_p\left(\frac{n!}{\gcd(D^n,n!)}\right)=\nu_p(n!)=\sum_{l\geq 1} \lfloor\frac{n}{p^l}\rfloor.
\]
Since this holds for every prime dividing $\frac{n!}{\gcd(D^n,n!)}$ the lemma follows.
\end{proof}

Let $D$ be a positive integer. For a fixed integer $s\geq 3D$, let us define
\[
	R_n^{(D)}(t):=D^{6(D-1)n}\cdot (n!)^{s-(3D-1)}\frac{\prod_{j=0}^{3Dn}(t-n+\frac{j}{D})}{\prod_{j=0}^n(t+j)^{s+1}}.
\]
The rational function $R_n^{(2)}$ coincides with the function studied in \cite{zudilin}. For $t\rightarrow \infty$ we have the asymptotic behaviour
\[
	R_n^{(D)}(t)=O\left( \frac{1}{t^{(s+1)(n+1)-3Dn-1}}\right).
\]
  Let us consider for $j\in\{1,...,D\}$ the following 'twisted' hypergeometric series
\[
	r_n^{(D,j)}:=\sum_{m=1}^\infty R^{(D)}_n\left(m+\frac{j}{D}\right).
\]
Let us write $a_{i,k}^{(D)}$ for the coefficients appearing in the partial-fraction decomposition of $R^{(D)}_n(t)$, i.e.
\[
	R_n^{(D)}(t)=\sum_{i=1}^s\sum_{k=0}^n \frac{a_{i,k}^{(D)}}{(t+k)^i}.
\]
Using the above lemma, we obtain:
\begin{lem}\label{lem3}
	The coefficients $a_{i,k}^{(D)}$ in the partial-fraction decomposition of $R_n^{(D)}(t)$ satisfy
	\[
		d_n^{s-i} a_{i,k}^{(D)}\in\ZZ.
	\]
	Further, we have
	\[
		a_{i,k}^{(D)}=(-1)^{i-1}(-1)^{nD}a_{i,n-k}^{(D)}.
	\]
\end{lem}
\begin{proof}
	In the case $D=2$, the proof is contained in \cite{zudilin}. The proof given there generalizes to more general integers $D\geq 2$: The symmetry
	\[
		a_{i,k}^{(D)}=(-1)^{i-1}(-1)^{nD}a_{i,n-k}^{(D)}.
	\]
	is an immediate consequence of the symmetry $R_n^{(D)}(-n-t)=(-1)^{nD}R_n(-t)$. For the integrality statement, let us consider the following decomposition of $R^{(D)}_n$ into more elementary rational functions:
	\begin{align*}
		R_n^{(D)}(t):&=D^{6(D-1)n}\cdot (n!)^{s-(3D-1)}\frac{\prod_{j=0}^{3Dn}(t-n+\frac{j}{D})}{\prod_{j=0}^n(t+j)^{s+1}}=\\
		&=\frac{(n!)^{s-(3D-1)}\cdot \prod_{j=1}^n(t-j)\prod_{j=1}^n(t+n+j)\prod_{i=1}^{D-1}D^{6n}\prod_{j=1}^{3n}(t-n+j-\frac{i}{D})}{\prod_{j=0}^n(t+j)^{s}}.
	\end{align*}
	Each of the following factors has an integral partial-fraction decomposition:
	\begin{align*}
			\frac{n!}{\prod_{j=0}^n(t+j)}&=\sum_{k=0}^n\frac{(-1)^k\binom{n}{k}}{(t+k)}\\
			\frac{\prod_{j=1}^n(t-j)}{\prod_{j=0}^n(t+j)}&=\sum_{k=0}^n\frac{(-1)^{n+k}\binom{n+k}{n}\binom{n}{k}}{(t+k)}\\
			\frac{\prod_{j=1}^n(t+n+j)}{\prod_{j=0}^n(t+j)}&=\sum_{k=0}^n\frac{(-1)^k\binom{2n-k}{n}\binom{n}{k}}{(t+k)}\\
	\end{align*}

	\begin{align*}
		\frac{D^{2n}\prod_{j=1}^n(t-j+\frac{i}{D})}{\prod_{j=0}^n(t+j)}&=\sum_{k=0}^n\frac{\frac{(-1)^{n+k}\prod_{j=k+1}^{n+k}(Dj-i)}{n!/\gcd(D^n,n!)}\binom{n}{k}\frac{D^n}{\gcd(D^n,n!)}}{(t+k)}\\
		\frac{D^{2n}\prod_{j=1}^n(t+j-\frac{i}{D})}{\prod_{j=0}^n(t+j)}&=\sum_{k=0}^n\frac{\frac{(-1)^{n+k}\prod_{j=1-k}^{n-k}(Dj+i)}{n!/\gcd(D^n,n!)}\binom{n}{k}\frac{D^n}{\gcd(D^n,n!)}}{(t+k)}\\
		\frac{D^{2n}\prod_{j=1}^n(t+n+j-\frac{i}{D})}{\prod_{j=0}^n(t+j)}&=\sum_{k=0}^n\frac{\frac{(-1)^{k}\prod_{j=n-k+1}^{2n-k}(Dj-i)}{n!/\gcd(D^n,n!)}\binom{n}{k}\frac{D^n}{\gcd(D^n,n!)}}{(t+k)}
	\end{align*}
The last three partial-fraction decompositions are integral by \cref{lem2}. The statement of the lemma follows now by applying \cref{lem1} to the product of simpler rational functions with integral partial-fraction decomposition.
\end{proof}
For later reference, let us record the following:
\begin{lem}\label{lem4}
The coefficients $a^{(D)}_{i,k}$ in the partial-fraction decomposition of $R^{(D)}_n$ satisfy for every $k\in\ZZ_{\geq0}$, $j\in\{1,...,D\}$ and $l\in\{0,...,n\}$ 
\[
	d_n^s \cdot \sum_{i=1}^s a_{i,k}^{(D)}\frac{1}{(l+\frac{j}{D})^i}\in \ZZ.
\]
\end{lem}
\begin{proof}
The rational function $R_n^{(D)}$ has zeros at $-2n,-2n+\frac{1}{D},-2n+\frac{2}{D},...,n-\frac{1}{D},n$. In particular,
\[
	0=R_n^{(D)}(-n+l+\frac{j}{D})=\sum_{i=1}^s\sum_{k=0}^n\frac{a_{i,k}^{(D)}}{(-n+k+l+\frac{j}{D})^i}.
\]
This can be reformulated as follows:
\[
	\sum_{i=1}^s a_{i,k}^{(D)}\frac{1}{(l+\frac{j}{D})^i}=-\sum_{i=1}^s\sum_{k=0}^{n-1}\frac{a_{i,k}^{(D)}}{(-n+k+l+\frac{j}{D})^i}.
\]
Let $p$ be a prime. Recall that the $p$-adic valuation $\nu_p$ is a homomorphism from $\QQ^\times$ to $\ZZ$ and satisfies the non-Archimedean triangle inequality
\[
	\nu_p(x+y)\geq \min(\nu_p(x),\nu_p(y)), \quad \forall x,y\in\QQ^\times
\]
with equality if $\nu_p(x)\neq \nu_p(y)$. Assume there were some prime $p$ with
\begin{equation}\label{eq_assumption1}
	\nu_p\left( \sum_{i=1}^s a_{i,k}^{(D)}\frac{1}{(l+\frac{j}{D})^i}\right)<-s\cdot\nu_p(d_n).
\end{equation}
From \cref{lem3} we deduce $\nu_p(a_{i,k}^{(D)})\geq -(s-i)\nu_p(d_n)$. Under our assumption, this implies
\[
	\nu_p\left(\frac{1}{(l+\frac{j}{D})^i}\right)<-i\nu_p(d_n)
\]
or equivalently
\[
	\nu_p\left(l+\frac{j}{D}\right)>\nu_p(d_n).
\]
By the non-Archimedean triangle inequality we obtain for $k\in\{0,...,n-1\}$
\[
	\nu_p\left(-n+k+l+\frac{j}{D}\right)=\nu_p(-n+k)\leq \nu_p(d_n).
\]
This implies
\begin{align*}
	\nu_p\left( \sum_{i=1}^s a_{i,k}^{(D)}\frac{1}{(l+\frac{j}{D})^i}\right)&=\nu_p\left(-\sum_{i=1}^s\sum_{k=0}^{n-1}\frac{a_{i,k}^{(D)}}{(-n+k+l+\frac{j}{D})^i}\right)\geq\\
	&\geq  \min_{i,k}\left\{\nu_p\left(a_{i,k}^{(D)}\frac{1}{(-n+kl+\frac{j}{D})^i}\right)\right\}\geq -\nu_p(d_n^s)
\end{align*}
a contradiction to \eqref{eq_assumption1}. Thus, we have
\[
	\nu_p\left( \sum_{i=1}^s a_{i,k}^{(D)}\frac{1}{(l+\frac{j}{D})^i}\right)\geq -s\cdot\nu_p(d_n)
\]
for every prime $p$ and deduce the desired integrality result
\[
	d_n^s \cdot \sum_{i=1}^s a_{i,k}^{(D)}\frac{1}{(l+\frac{j}{D})^i}\in \ZZ.
\]
\end{proof}
Let us write
\[
	\zeta(s,q):=\sum_{n=0}^\infty \frac{1}{(n+q)^s},\quad \Re(s)>1,q\neq 0
\]
for the Hurwitz zeta function.
\begin{lem}\label{lem5}
Let $n,D\geq1$ be integers with $nD$ even. For $j\in\{1,...,D\}$, we have	
\[
	r^{(D,j)}_n=\sum_{\substack{i=3 \\ i \text{ odd}}}^s a^{(D)}_i\zeta\left(i,\frac{j}{D}\right)+a^{(D,j)}_0
\]
with 
\[
	d_n^{s-i}a^{(D)}_i\in \ZZ \text{ for } i=3,5,...,s\quad \text{ and }\quad  d_n^s a_0^{(D,j)} \text{ for } j\in\{0,...,D-1\}.
\]
\end{lem}
\begin{proof}
	We proceed as in \cite[Lemma 3]{zudilin}: We introduce an auxiliary parameter $z>0$:
	\begin{align*}
		r_n^{(D,j)}(z)&=\sum_{m=1}^\infty R_n\left(m+\frac{j}{D}\right)z^m=\sum_{m=1}^\infty \sum_{i=1}^s\sum_{k=0}^n \frac{a^{(D)}_{i,k}z^m}{(m+k+\frac{j}{D})^i}=\\
		&= \sum_{i=1}^s\sum_{k=0}^na^{(D)}_{i,k}z^{-k}\sum_{m=1}^\infty \frac{z^{m+k}}{(m+k+\frac{j}{D})^i}=\\
		&= \sum_{i=1}^s\sum_{k=0}^na^{(D)}_{i,k}z^{-k}\left(\Phi(z,\frac{j}{D},i)-\sum_{l=0}^k \frac{z^{l}}{(l+\frac{j}{D})^i}  \right)
	\end{align*}
	Here,
	\[
		\Phi(z,\alpha,s):=\sum_{n=0}^\infty \frac{z^n}{(n+\alpha)^s}
	\]
	is the Lerch zeta function. The Lerch zeta function converges for $s\geq 2$ at $z=1$ and $\Phi(1,\frac{j}{D},i)=\zeta(i,\frac{j}{D})$. For $i=1$ the limit $z\nearrow 1$ does not exist. Thus, by taking the limit $z\nearrow 1$ in the above equality, we conclude
	\[
		\sum_{k=0}^na^{(D)}_{1,k}=\lim_{z\nearrow 1} \sum_{k=0}^na^{(D)}_{1,k}z^{-k}=0.
	\]
	Let us define
	\[
		a^{(D)}_i:=\sum_{k=0}^n a_{i,k}^{(D)}.
	\]
	Since we assumed $Dn$ to be even,  we have the symmetry
	\[
		a_{i,k}^{(D)}=(-1)^{i-1}a_{i,n-k}^{(D)}
	\]
	and hence  $a^{(D)}_i=0$ for $i$ even. By \cref{lem3}  and \cref{lem4} we have 
	\begin{align*}
		a_0^{(D,j)}:=-\sum_{i=1}^s\sum_{k=0}^n\sum_{l=0}^k a^{(D)}_{i,k} \frac{z^{l}}{(l+\frac{j}{D})^i}  \in d_n^{-s}\ZZ,\quad d_n^{s-i}a^{(D)}_i&\in\ZZ
	\end{align*}
	as desired.
\end{proof}

\section{Asymptotic behaviour}
Let us analyze the asymptotic behaviour of $r_n^{(D,j)}$. We follow closely the argument in \cite{zudilin}.
Because $R_n^{(D)}(t)$ vanishes at $1+\frac{j}{D},2+\frac{j}{D},...,n-1+\frac{j}{D}$ we have for $j\in\{1,...,D\}$:
\[
	r^{(D,j)}_n=\sum_{m=1}^\infty R^{(D)}_n(m+\frac{j}{D})=\sum_{k=0}^\infty c_k^{(D,j)}
\]
with
\[
	c_k^{(D,j)}:=R^{(D)}_n\left(n+k+\frac{j}{D}\right).
\]

\begin{lem}\label{lem6}
Let $s\geq 3D+1 $ be odd and $j,\tilde{j}\in\{1,...,D\}$. We have the asymptotic behaviour
\[
	\lim_{n\rightarrow \infty} \sqrt[n]{r_n^{(D,j)}}=g_D(x_0)\quad \text{and} \quad \lim_{n\rightarrow \infty}\frac{r^{D,j}_n}{r^{D,\tilde{j}}_n}=1
\]
where
\[
g_D(x)=D^{6(D-1)}\frac{(3+x)^{3D}(1+x)^{s+1}}{(2+x)^{2(s+1)}}
\]
and $x_0$ is the unique positive zero of the polynomial
\[
	(x+3)^D(x+1)^{s+1}-x^D(x+2)^{s+1}.
\]
\end{lem}
\begin{proof}Again, we follow closely the proof in the case $D=2$ given in \cite[Lemma 4]{zudilin}:	The definition of $c_k^{(D,j)}$ gives
	\[
		c_k^{(D,j)}=D^{6(D-1)n}(n!)^{s-(3D-1)}\frac{\prod_{l=0}^{3Dn}(k+\frac{j+l}{D})}{\prod_{l=0}^{n}(n+k+l+\frac{j}{D})^{(s+1)}}
	\]
	which allows us to deduce the asymptotics
	\[
		\frac{c^{(D,j)}_{k+1}}{c^{(D,j)}_k}	=\left(\prod_{l=1}^D\frac{(k+3n+\frac{j+l}{D})}{(k+\frac{j+l-1}{D})}\right)\left( \frac{k+n+\frac{j}{D}}{k+2n+1+\frac{j}{D}}\right)^{s+1}\sim f_D\left(\frac{k}{n}\right),\quad \text{ as }n\rightarrow\infty,
	\]
	with $f_D(x)=\left(\frac{x+3}{x}\right)^D\left(\frac{x+1}{x+2}\right)^{s+1}$. The logarithmic derivative of $f_D(x)$ is
	\[
		\frac{f'_D(x)}{f_D(x)}=\frac{D}{x+3}-\frac{D}{x}+(s+1)\left(\frac{1}{x+1}-\frac{1}{x+2} \right)=\frac{(s+1-3D)x^2+(3s+3-9D)x-6D}{x(x+1)(x+2)(x+3)}.
	\]
	The quadratic polynomial in the numerator of the latter fraction has a unique positive zero $x_1$. Thus, $f_D$ monoton decreases from $+\infty$ to $f_D(x_1)$ on $(0,x_1)$ and then monoton increases on $(x_1,\infty)$ from $f_D(x_1)$ to $1$. In particular, there is a unique value $x_0$ with $f_D(x_0)=1$. Since $f_D(1)=4^D\cdot(2/3)^{(s+1)}<1$ we have $0<x_0<1$. Thus, the sequence $(c_k^{(D,j)})_{k\geq 1}$ attains its maximal values around $x_0n$. We follow the asymptotic analysis of series in de Bruijn's book \cite[\S 3.4]{deBruijn}. Using Stirling's formula
\[
	\Gamma(x)=\sqrt{\frac{2\pi}{x}}\left(\frac{x}{e}\right)^x\left(1+O\left( \frac{1}{x} \right)\right),\quad x\rightarrow \infty.
\]
	one can prove, that there is some $C>0$ such that
	 \[
	 	\lim_{n\rightarrow\infty} \frac{R_n^{(D)}\left( n+x_0n\pm t\cdot \sqrt{n} \right)}{R_n^{(D)}\left( n+x_0n\right)}=\exp\left( -\frac{C}{t^2} \right),\quad t\in\mathbb{R}.
	 \]
	Recall, that $c_{k}^{(D,j)}$ has been defined as the evaluation of $R_n^{(D)}$ at $n+k+j/D$. We deduce the existence of $\gamma>0$ with
	\[
		c_{k}^{(D,j)}<\frac{1}{2}\max_k\{c_{k}^{(D,j)}\},\quad \text{ for $k\in\ZZ$ with } |k-x_0n|>\gamma \sqrt{n}.
	\]
	The total contribution of the terms outside $[x_0n-\gamma\sqrt{n}, x_0n+\gamma\sqrt{n}]$ is relatively small. Thus the asymptotic behaviour of $r_n^{(D,j)}=\sum_{k=0}^\infty c_k^{(D,j)}$ is determined by the asymptotics of $c_k$ for $k$ in $\{ x_0n-\gamma\sqrt{n},x_0n+\gamma\sqrt{n} \}$. Let us define $k_0(n):= \lfloor x_0n \rfloor$. Using Stirling's formula, we compute
	\begin{align*}
		\lim_{n\rightarrow\infty } \sqrt[n]{r_n^{(D,j)}}&=\lim_{n\rightarrow\infty } \sqrt[n]{c_{k_0(n)}^{(D,j)}}=\lim_{n\rightarrow\infty } \Big[D^{3(D-2)n}\left(\frac{n}{e}\right)^{(s-3D+1)n} \times \\
		&\quad \times \left(\frac{3Dn+Dk_0(n)+j}{e}\right)^{3Dn+Dk_0(n)+j}\left(\frac{k_0(n)+n+\frac{j}{D}}{e}\right)^{(s+1)(k_0(n)+n+\frac{j}{D})} \times\\
		&\quad  \times \left(\frac{e}{k_0(n)+2n+\frac{j}{D}+1}\right)^{(s+1)(k_0(n)+2n+\frac{j}{D}+1)} \left(\frac{e}{Dk_0(n)+j-1}\right)^{Dk_0(n)+j-1}  \Big]^{1/n}=\\
		&=D^{3(D-2)}\frac{(3D+Dx_0)^{3D}(3D+Dx_0)^{Dx_0}(1+x_0)^{(s+1)(x_0+1)}}{(Dx_0)^{Dx_0}(2+x_0)^{(s+1)(2+x_0)}}=\\
		&=D^{3(D-2)}D^{3D}\cdot \frac{(3+x_0)^{3D}(1+x_0)^{(s+1)}}{(2+x_0)^{2(s+1)}}f_D(x_0)^{x_0}=\\
		&=D^{6(D-1)}\cdot \frac{(3+x_0)^{3D}(1+x_0)^{(s+1)}}{(2+x_0)^{2(s+1)}}=g(x_0)
	\end{align*}
	Let us compare for $n\rightarrow\infty$ the quotients
\begin{align*}
	\frac{c_k^{(D,j)}}{c_k^{(D,D)}}&=\left(\prod_{l=0}^{D-j-1}\frac{k+3n+\frac{j+l+1}{D}}{k+\frac{j+l}{D}}\right)\left(\frac{\Gamma(2n+k+1+\frac{j}{D})\Gamma(n+k+2)}{\Gamma(2n+k+2)\Gamma(n+k+1+\frac{j}{D})}\right)^{s+1}\sim\\
	&\sim \left(\frac{3n+k}{2k}\right)^{D-j}\left( \frac{n+k+2}{2n+k+2}\right)^{\frac{D-j}{D}(s+1)}.
\end{align*}
	This allows us to compare $r_n^{(D,j)}$ for different values of $j\in\{1,...,D\}$:
	\begin{align*}
		\lim_{n\rightarrow\infty } \frac{r_n^{(D,j)}}{r_n^{(D,D)}}&=\lim_{n\rightarrow\infty } \frac{c_{k_0(n)}^{(D,j)}}{c_{k_0(n)}^{(D,D)}}=\lim_{n\rightarrow\infty } \left( \frac{3n+k_0(n)}{k_0(n)} \right)^{D-j}\cdot \left( \frac{n+k_0(n)+2}{2n+k_0(n)+2} \right)^{\frac{D-j}{D}(s+1)}=\\
		&=f_D(x_0)^{\frac{D-j}{D}}=1.
	\end{align*}
	Since this holds for any $j\in\{1,...,D\}$ we get $\lim_{n\rightarrow \infty}\frac{r^{D,j}_n}{r^{D,\tilde{j}}_n}=1$.
\end{proof}

\section{Conclusion}
The aim of this section is to prove that there are infinitely many irrational zeta values of the form $\zeta(2n+1)$. We will use the above linear forms in the Hurwitz zeta function together with the formula
\[
	\sum_{j=1}^{D-1}\zeta\left(i,\frac{j}{D}\right)=\sum_{j=1}^{D-1}\sum_{n=0}^\infty \frac{D^i}{(Dn+j)^i}= (D^i-1)\zeta(i)
\]
to eliminate certain odd zeta values. In particular, we get for each divisor $d|D$
\[
	\left(\left(\frac{D}{d}\right)^i-1\right)\zeta(i)=\sum_{j=1}^{(\frac{D}{d})-1}\zeta\left(i,\frac{jd}{D}\right).
\]
This motivates the following definition
\[
	\hat{r}_n^{(D,d)}:=\sum_{j=1}^{(\frac{D}{d})-1} r_n^{(D,dj)}.
\]
With the above observation, we obtain for every divisor $d|D$ linear forms in odd zeta values:
\[
	\hat{r}_n^{(D,d)}=\sum_{\substack{i=3 \\ i \text{ odd}}}^s a^{(D)}_i\left(\left(\frac{D}{d}\right)^i-1\right)\zeta(i)+\hat{a}^{(D,d)}_0
\]
with $\hat{a}^{(D,d)}_0=\sum_{j=1}^{(\frac{D}{d})-1} a_0^{(D,dj)}$. Observe, that the coefficients $a_i^{(D)}$ do not depend on $d|D$. Choosing $D>1$ with many divisors and using that the matrix
\[
	\left(\left(\frac{D}{d}\right)^i-1\right)_{\substack{i=3,...,s \\ d|D, d\neq D}}
\]
has full rank, allows us to eliminate arbitrary many odd zeta values in such a linear form. In order to guarantee convergence, it is necessary to choose $s>>1$ appropriately. We obtain:
\begin{thm}
For each integer $m\geq 0$, there exists an integer $N_m>m$ such that every subset of
\[
	\{\zeta(3),\zeta(5),...,\zeta(2N_m+1)\}
\]
with $N_m-m$ elements contains at least one irrational number.
\end{thm}
\begin{proof}
Let us choose $D=2^{m+1}$. Let us write $x_{0}(s)$ for the unique positive solution of $f_{D}(x_0,s)=1$ where
\[
	f_{D}(x,s)=\left(\frac{x+3}{x}\right)^D\left( \frac{x+1}{x+2} \right)^{s+1}.
\]
We have $f_{D}(4\cdot 2^{-\frac{s+1}{D}},s)<1<f_{D}(2^{-\frac{s+1}{D}},s)$ for $s>>1$. By the discussion in the proof of \cref{lem6} we obtain
\[
	2^{-\frac{s+1}{D}}<x_{0}(s)<4\cdot 2^{-\frac{s+1}{D}}.
\]
In particular, we deduce that
\[
	g_{D}(x_0(s),s)=O\left(\frac{1}{4^s}\right),\quad s\rightarrow \infty
\]
where $g_{D}(x,s)$ is the polynomial
\[
g_{D}(x,s)=D^{6(D-1)}\frac{(3+x)^{3D}(1+x)^{s+1}}{(2+x)^{2(s+1)}}
\]
appearing in the statement of \cref{lem6}. Since $e^{-1}>\frac{1}{4}$, we can choose a large enough integer $N_m$ with $g_{D}(x_0(s),s)<e^{-s}$ for $s=2N_m+1$. The prime number theorem guaranties $\lim_{n\rightarrow \infty} \sqrt[n]{d_n}=e$. With the choice $s:=2N_m+1$, \cref{lem6} implies the convergence of
\[
	d_n^s\cdot \hat{r}_n^{(D,d)} \rightarrow 0,\text{ as }n\rightarrow\infty.
\]
Thus for the above choice of $s$, $d_n^s\cdot \hat{r}_n^{(D,d)}$ provides integral linear forms in odd zeta values converging to zero. Let
\[
	J\subset \{ 3,5,...,2N_m+1 \}
\]
be a subset with $N_m-m$ elements. Let us choose some $j\in J$ and write
\[
	\{j\}\cup (\{ 3,5,...,2N_m+1 \}\setminus J)=\{ i_1,...,i_{m+1} \}
\]
for its complement with $i_1<i_2<...<i_{m+1}$. Let $l$ be the index corresponding to $j$, i.e. $j=i_l$. By the choice $D=2^{m+1}$, every divisor $d|D$ is of the form $d=2^{m+1-\beta}$ for $\beta=0,...,m+1$. The matrix
\[
	M=\left( \left(\frac{D}{d}\right)^{i_\alpha}-1 \right)_{\alpha,d}=(2^{i_\alpha\cdot \beta}-1)_{\alpha,\beta=1,...,m+1}
\]
is invertible. Define the following vector with integral entries
\[
	w=(w_1,...,w_{m+1}):=\det(M)\cdot e_l\cdot M^{-1}\in \ZZ^{m+1}
\]
with $e_l=(0,...,0,1,...,0)$ the $l$-th unit vector. Then
\[
	d_n^s\hat{r}_n:=(-1)^{l-1}d_n^s\sum_{k=1}^{m+1}w_k \hat{r}_n^{(2^{m+1},2^{m+1-k})}
\]
is an integral linear combination of the odd zeta values in $J$ and tends to zero as $n$ goes to infinity. A straight-forward computation shows
\[
	(-1)^{l-1}\cdot \sum_{k=1}^{m+1}w_k \cdot \left(2^{k}-1\right)>0
\]
which implies the positivity of the involved linear forms:
\[
	\lim_{n\rightarrow \infty} \frac{\hat{r}_n}{r_n^{(D,D)}}=\lim_{n\rightarrow \infty} (-1)^{l-1}\sum_{k=1}^{m+1}w_k\sum_{j=1}^{2^k-1}\frac{r_n^{(2^{m+1},2^{m+1-k}j)}}{r_n^{(2^{m+1},2^{m+1})}}=(-1)^{l-1}\sum_{k=1}^{m+1}w_k\left(2^k-1\right)>0.
\]
Thus we obtain a sequence of non-zero integral linear combinations of 
\[
	\{ \zeta(j)\mid j\in J \}
\]
tending to zero. If all of the zeta values $\zeta(j)$ where rational, we could multiply by there common denominator and would get a sequence of non-zero integers converging to zero - an obvious contradiction.
\end{proof}

\begin{cor}
	For each $m\geq 1$ there is an integer $N_m>m$ such that at least one of the numbers
	\[
		\{\zeta(2m+1),\zeta(2m+3)...,\zeta(2N_m+1)\}
	\]
	is irrational. In particular, there are infinitely many odd zeta values, which are irrational.
\end{cor}

\begin{rem}
We have chosen the rational functions $R_n^{(D)}$ in a preferably simple way. For improving convergence properties, there are probably much better choices.
\end{rem}

\bibliographystyle{amsalpha} 
\bibliography{irrational_zeta}
\end{document}